\documentclass{birkjour} 
\usepackage{amsopn}
\usepackage{amsmath,amsthm,amssymb,color,enumerate}
\usepackage{prettyref}

\usepackage{pdfsync}
\newcommand{\nc}{\newcommand}

 \nc{\aff}{\mathfrak{aff} } \nc{\bb}{\mathfrak{b}}
\nc{\cc}{\mathfrak{c} }  \nc{\dd}{\mathfrak{d} }
 \nc{\ggo}{\mathfrak{g} }
 \nc{\hh}{\mathfrak{h} }  \nc{\ii}{\mathfrak{i} }
 \nc{\jj}{\mathfrak{j} }  \nc{\kk}{\mathfrak{k} }
\nc{\mm}{\mathfrak{m} }   \nc{\nn}{\mathfrak{n} }
\nc{\pp}{\mathfrak{p} }  \nc{\rr}{\mathfrak{r} } \nc{\sg}{\mathfrak{s} }
 \nc{\sog}{\mathfrak{so} }  \nc{\spg}{\mathfrak{sp}}
 \nc{\sug}{\mathfrak{su} }  \nc{\slg}{\mathfrak{sl}}
 \nc{\tg}{\mathfrak{t} }  \nc{\uu}{\mathfrak{u} }
 \nc{\vv}{\mathfrak{v} } \nc{\ww}{\mathfrak{w} }
 \nc{\zz}{\mathfrak{z} }

 \nc{\ggob}{\overline{\mathfrak{g}}}

\nc{\glg}{\mathfrak{gl} }

\nc{\pca}{\mathcal{P}} \nc{\nca}{\mathcal{N}}

 \nc{\vp}{\varphi} \nc{\ddt}{\frac{{\rm d}}{{\rm d}t}}
 \nc{\la}{\langle} \nc{\ra}{\rangle}

 \nc{\SO}{{\sf SO}} \nc{\Spe}{{\sf Sp}} \nc{\Sl}{{\sf Sl}}
 \nc{\SU}{{\sf SU}} \nc{\Or}{{\sf O}} \nc{\U}{{\sf U}}
 \nc{\Gl}{{\sf Gl}} \nc{\Se}{{\sf S}} \nc{\Cl}{{\sf Cl}}
 \nc{\Spin}{{\sf Spin}} \nc{\Pin}{{\sf Pin}}

 \nc{\RR}{{\mathbb R}} \nc{\HH}{{\mathbb H}} \nc{\CC}{{\mathbb C}}
 \nc{\ZZ}{{\mathbb Z}} \nc{\FF}{{\mathbb F}} \nc{\NN}{{\mathbb N}}
 \nc{\GG}{{\mathbb G}} \nc{\JJ}{{\mathbb J}} \nc{\II}{{\mathbb I}}
 \nc{\KK}{{\mathbb K}} \nc{\DD}{{\mathbb D}}

 \nc{\ad}{\operatorname{ad}} \nc{\Ad}{\operatorname{Ad}}
 \nc{\coad}{\operatorname{coad}} \nc{\ct}{\operatorname{T}}
 \nc{\rank}{\operatorname{rank}} \nc{\Irr}{\operatorname{Irr}}
 \nc{\End}{\operatorname{End}} \nc{\Aut}{\operatorname{Aut}}
 \nc{\Inn}{\operatorname{Inn}} \nc{\Der}{\operatorname{Der}}
 \nc{\Dera}{\operatorname{Dera}} \nc{\Auto}{\operatorname{Auto}}
 \nc{\GL}{\operatorname{GL}}
 \nc{\SL}{\operatorname{SL}}
 \nc{\coord}{\operatorname{coord}}
 \renewcommand{\span}{\operatorname{span}}
 \nc{\codim}{\operatorname{codim}}

 \theoremstyle{plain}
 \newtheorem{teo}{Theorem}[section]
 
 \newtheorem{pro}[teo]{Proposition}
 \newtheorem{cor}[teo]{Corollary}
 \newtheorem{lm}[teo]{Lemma}

 \renewenvironment{proof}{\noindent \emph{Proof.}}{\hfill $\blacksquare$}

 \theoremstyle{remark}
 
 \newtheorem{rem}[teo]{Remark}
 \newtheorem{defi}[teo]{Definition}
 
 \newtheorem{example}[teo]{Example}

 \newcommand{\R}{\mathbb R}

\newcommand{\Z}{\mathbb Z}

\newcommand{\mg}{\mathfrak g }
\newcommand{\mb}{\mathfrak b }

\newcommand{\mn}{\mathfrak n }
\newcommand{\mz}{\mathfrak z }

\newcommand{\mv}{\mathfrak v }
\newcommand{\mh}{\mathfrak h }

\newcommand{\mgg}{\mathfrak g }
\newcommand{\mpp}{\mathfrak p }

\renewcommand{\sl}{\mathfrak{sl} }

\nc{\rmc}{\textrm{\rm C}}
\nc{\rad}{\textrm{\rm Rad}}

\newcommand{\bil}{\left\langle \;,\;\right\rangle}
\newcommand{\lela}{\left \langle}
\newcommand{\rira}{\right \rangle}
\newcommand{\lra}{\longrightarrow}

\newcommand{\bs}{\backslash}
 
\begin{document}
\title[Lie algebras admitting symmetric, invariant and nondegenerate bilinear forms]{Lie algebras admitting symmetric, invariant and nondegenerate bilinear forms}

\author{Viviana del Barco}
\email{delbarc@fceia.unr.edu.ar}
\address{CONICET -- Universidad Nacional de Rosario.\\ ECEN-FCEIA, Depto. de Matem\'a\-tica, Av. Pellegrini 250, 2000 Rosario, Santa Fe, Argentina.}

\date{\today}

\begin{abstract} We present structural properties of Lie algebras admitting symmetric, invariant and nondegenerate bilinear forms. We show that these properties are not satisfied by nilradicals of parabolic subalgebras of real split forms of complex simple Lie algebras, neither by 2-step nilpotent Lie algebras associated with graphs, with only few exceptions. These rare cases are, essentially, abelian Lie algebras, the free 3-step nilpotent Lie algebra on 2-generators and the free 2-step nilpotent Lie algebra on 3-generators.
\end{abstract}

\thanks{This work was partially supported by SECyT-UNR, CONICET and ANPCyT. \\
{\it (2010) Mathematics Subject Classification}:   17B05, 17B20, 17B22, 17B30, 53C50.
}


\maketitle

\section{Introduction} 
In this paper we consider Lie algebras $\mgg$ admitting ad-invariant metrics, that is, symmetric bilinear forms which are nondegenerate and for which inner derivations are skew-symmetric. Several authors considered these Lie algebras because of their re\-levance in different areas of mathematics and physics \cite{BK,Bo,Co,FiS,HK,Ka,Kos,NW}. Lie algebras carrying ad-invariant metrics are also known as metric, self-dual or quadratic Lie algebras, depending on the context. Recent compilation works on advances and applications are \cite{KaOl2,Ov}.

In this paper we establish structural properties of Lie algebras admitting ad-invariant metrics. Those properties become a set of obstructions useful when deciding whether a given Lie algebra can be equipped with such a metric or not. The main result of this work is conclusive (see Theorems \ref{teo.classgraph} and \ref{teo.classif}):
\begin{center}
{\em Let $\mn$ be a nilpotent Lie algebra admitting an ad-invariant metric and such that either $\mn$ is a 2-step nilpotent Lie algebra associated with a graph, or $\mn$ is the nilradical of a parabolic subalgebra of a real split form of a complex simple Lie algebra. Then $\mn$ is abelian, $\mn$ is isomorphic to $\mn_{2,3}$ or it is isomorphic to a direct sum of copies of $\mn_{3,2}$  and an abelian factor.} 
\end{center}
Here $\mn_{p,k}$ denotes the free $k$-step nilpotent Lie algebra on $p$-generators.

It is well known that semisimple Lie algebras and abelian Lie algebras carry ad-invariant metrics.  Nonsemisimple Lie algebras may also admit such metrics, examples of which are solvable or nilpotent Lie algebras \cite{Me,Pe}. But there are restrictions, for instance, solvable Lie algebras with null center or nilpotent Lie algebras with one-dimensional center cannot be equipped with an ad-invariant metric. 

Lie algebras carrying ad-invariant metrics can be constructed by the double extension procedure; any such Lie algebra is obtained in this way when it is nonsimple and of dimension $\geq 2$ \cite{FS,Ka,MR}. The production of examples is guaranteed by this method and it also allowed to obtain algebraic and geometric properties of these Lie algebras and their associated Lie groups with  bi-invariant metrics. Yet, one may reach isomorphic Lie algebras starting from two nonisomorphic ones, making difficult the classification of Lie algebras admitting ad-invariant metrics by means of this extension procedure.

Still, whether a given nonsemisimple Lie algebra admits an ad-invariant metric is an open question. 
Partial results when restricting to particular families of Lie algebras can be found in \cite{BK,dBOV,Me2,KaOl}. 

Our work here contributes with general Lie-algebraic structure facts that Lie algebras admitting ad-invariant metrics satisfy. When considering the families of nilpotent Lie algebras mentioned above we see that these conditions are quite restrictive, giving rise to the conclusive result stated before.

The organization of this work is as follows. Section \ref{sec.metric} introduces notation and structural properties of Lie algebras admitting ad-invariant metrics.  We apply the obstructions obtained here to particular families of nilpotent Lie algebras. Section \ref{sec.graphs} is devoted to study which Lie algebras associated with graphs admit ad-invariant metrics. Meanwhile Section \ref{sec.nilrad} reviews known properties of the central descending series of nilradicals of parabolic subalgebras of split forms of complex simple Lie algebras, and also we give particular behavior of roots in the last step of its lower central series (Proposition \ref{progammadecomp}). This is a key step to narrow the possible Lie algebras admitting ad-invariant metrics in this family and achieve the full classification.

\section{Invariant metrics on Lie algebras}\label{sec.metric}
In this section we introduce notations and basic facts about the structure of Lie algebras admitting ad-invariant metrics. 

Let $\mgg$ be a real Lie algebra of dimension $m$. Given a subspace $\mv$ of $\mgg$ we consider the series $\{\rmc^j(\mv)\}$ and $\{\rmc_j(\mv)\}$ where for all $j\geq 0$ they are given by
\begin{eqnarray*}
&&\rmc^0(\mv)=\mv,\quad  \rmc^j(\mv)=[\mg,\rmc^{j-1}(\mv)],\;\; j\geq 1\quad \mbox{ and } \\
&&\rmc_0(\mv)=0,\quad \rmc_1(\mv)=\{X\in\mg:\,[X,\mv]=0\},\\
&&\rmc_j(\mv)=\{X\in\mg:\,[X,\mg]\subset \rmc_{j-1}(\mv)\},\;\; j\geq 2.\end{eqnarray*}
The subalgebra $\rmc_1(\mv)$ is the centralizer $\mz(\mv)$ of $\mv$ in $\mgg$. Given $X\in \mgg$, we denote $\mz(\R X)$ simply as $\mz(X)$. Notice that $\rmc^{j}(\mv)\subset \rmc^{j-1}(\mv)$ (resp.  $\rmc^{j}(\mv)\subset \rmc^{j+1}(\mv)$) for all $j\geq 0$ if and only if $\mv$ (resp. $\mz(\mv)$) is an ideal.

The particular series $\{\rmc^j(\mg)\}$ and $\{\rmc_j(\mg)\}$ are, respectively, the lower central series and the upper central series of $\mgg$. We notice that $\rmc^1(\mgg)=[\mg,\mg]$ is the commutator of $\mg$ and $\rmc_1(\mgg)$ is the center of $\mg$ which we denote by $\mz$.

Let $\la \,,\, \ra:\mgg\times \mgg\lra \R$ be a symmetric bilinear form on $\mgg$. The orthogonal of a subspace $\mv$ of $\mg$ is the vector subspace of $\mgg$ given by $$\mv^\bot=\{Y\in\mgg:\; \lela  X,Y\rira=0  \mbox{ for all }X\in\mv\}.$$
We denote $\rad=\mgg^\bot$, that is, \begin{equation}\label{rad}
\rad =\{Y\in\mgg: \lela X,Y\rira=0 \mbox{ for all } X\in \mgg\}.
\end{equation}
In particular, $\bil$ is nondegenerate if and only if $\rad=0$.

A symmetric bilinear form $\bil$ is called invariant when
\begin{equation}\label{adme}
\la [X, Y],  Z\ra + \la Y, [X, Z]\ra = 0 \qquad \mbox{ for all }X, Y, Z \in  \mgg.
\end{equation}
If $\bil$ is an invariant symmetric bilinear form then $\rad$ is an ideal of $\mgg$.

\begin{defi} An {\em ad-invariant metric} on a Lie algebra $\mgg$ is a nondegenerate symmetric invariant bilinear form $\bil$.\end{defi}

When $\mgg$ can be endowed with an ad-invariant metric the series associated to a subspace $\mv\subset \mgg$ are related.

\begin{pro} \label{pro.adinv}\label{padinv.ort} \label{padinv.dim} Let $\mgg$ be a Lie algebra endowed with an ad-invariant metric $\bil$ and let $\mv$ be a subspace of $\mgg$. Then for all $j\geq 1$
\begin{equation}\label{eq.ort}
\rmc^{j}(\mv)^\bot=\rmc_j(\mv)
\end{equation}
and therefore 
\begin{equation}
\dim\rmc_j(\mv)+\dim \rmc^j(\mv)= \dim \mgg.\label{eq.dim}
\end{equation}
\end{pro}
\begin{proof} We proceed by induction on $j\geq 1$. Set $X\in \rmc^1(\mv)$ and $Y\in\rmc_1(\mv)=\mz(\mv)$, so $X=[U,V]$ for some $V\in\mv$ and we have $\lela X, Y\rira=\lela[U,V], Y\rira=\lela U,[V,Y]\rira=0$. Thus $\rmc_1(\mv)\subset\rmc^{1}(\mv)^\bot$. Now let $X\in\rmc^1(\mv)^\bot$ and let $V\in\mv$. For any $Y\in\mgg$ we have 
$\lela [X,V],Y\rira=\lela X,[V,Y]\rira=0$, which implies $[X,V]=0$ since $\bil$ is nondegenerate. Thus $X\in\rmc_1(\mv)$.

Assume the statement holds for $j-1$ and fix $X\in \rmc^{j}(\mv)$, so $X=[U,V]$ with $V\in\rmc^{j-1}(\mv)=\rmc_{j-1}(\mv)^\bot$. Given $Y\in\rmc_{j}(\mv)$ we have $\lela X, Y\rira=\lela[U,V], Y\rira=-\lela V,[U,Y]\rira$ where now $[U,Y]\in\rmc_{j-1}(\mv)$ since $Y\in\rmc_j(\mv)$, thus $\lela X,Y\rira=0$ and $\rmc_j(\mv)\subset\rmc^{j}(\mv)^\bot$.
Take $X\in\rmc^j(\mv)^\bot$, then for any $V\in\rmc^{j-1}(\mv)$ and $Y\in\mgg$ we have $\lela [X,Y],V\rira=\lela X, [Y,V]\rira$ where $[Y,V]\in\rmc^j(\mgg)$. So $\lela [X,Y],V\rira=0$ and nondegeneracy implies $[X,\mg]\subset\rmc^{j-1}(\mv)^\bot\subset\rmc_{j-1}(\mv)$. Thus $X\in\rmc_j(\mv)$ and $\rmc^{j}(\mv)^\bot=\rmc_j(\mv)$ for all $j\geq 1$. Formula \eqref{eq.dim} holds since $\bil$ is nondegenerate.
\end{proof}\smallskip

The result in Proposition  \ref{pro.adinv} for the particular case $\mv=\mgg$, that is, the lower and upper central series of $\mgg$, was already known (see for instance \cite{FS,HK}). It is worth noticing that \eqref{eq.dim} may hold for $\mgg$ but fail for $\mv\neq \mgg$. For instance take $\mgg=\span_\R\{e_i:i=1,\ldots,8\}$ with nonzero bracket relations $[e_1,e_2]=e_7$, $[e_2,e_3]=e_8$,  $[e_3,e_4]=e_5$ and $[e_1,e_4]=e_6$; and $\mv=\span_\R\{e_1,e_2\}$. 

Under the hypothesis of Proposition \ref{pro.adinv} one has
\begin{equation}\label{eq.comm}
[\rmc_{j+1}(\mv),\mgg]\subset \rmc^{j}(\mv)^\bot\quad\mbox{  for all }j\geq 0\end{equation} since $[\rmc_{j+1}(\mv),\mgg]\subset \rmc_j(\mgg)$ by definition and $\rmc_j(\mgg)\subset \rmc^{j}(\mv)^\bot$ by Eq. \eqref{eq.ort}.

Let $\mgg$ be a Lie algebra admitting ad-invariant metrics. Then Eq. \eqref{eq.comm} implies \begin{equation}\label{eq.cap}  \bigcap_{X\in\mgg}[\mz(X),\mgg]\subset \bigcap_{X\in\mgg} (\R X)^\bot=\rad,\end{equation} where the orthogonal of $\R X$ is taken with respect to any nondegenerate symmetric bilinear form.
Assume now that $\rmc^1(\mgg)\neq \mgg$ and let $\mgg=\mv_1\oplus\cdots\oplus\mv_s\oplus\rmc^1(\mgg)$ be a vector space direct sum of $\mgg$ with $\mv_i\neq 0$ for $i=1,\ldots,s$. Related to this decomposition we define the following ideal
\begin{equation}\label{center}
\Theta_{\mv_1,\ldots,\mv_s}=\mz\cap \bigcap_{i=1,\ldots,s}[ \mz(\mv_i),\mgg].\end{equation}
This ideal is a subspace of $\rad$, independently of the choice of $\mv_i$. Indeed Eqns. \eqref{eq.ort} and \eqref{eq.comm} imply $\mz\subset \rmc^1(\mgg)^\bot$ and $[ \mz(\mv_i),\mgg]\subset \mv_i^\bot$, respectively. Thus
\begin{equation}
\Theta_{\mv_1,\ldots,\mv_s} \subset \rmc^1(\mgg)^\bot\cap \bigcap_{i=1,\ldots,s}\mv_i^\bot =\rad.\label{eq.theta}\end{equation} 
These facts account to the following.

\begin{pro} \label{pro.cap} Let $\mgg$ be a Lie algebra admitting ad-invariant metrics. Then \begin{equation}\label{padinv.cap}\bigcap_{X\in\mgg}[\mz(X),\mgg]=0.\end{equation} If moreover $\rmc^1(\mgg)\neq \mgg$ then  \begin{equation}\label{padinv.theta}\Theta_{\mv_1,\ldots,\mv_s}=0\end{equation} for any nontrivial decomposition $\mgg=\mv_1\oplus\cdots\oplus\mv_s\oplus\rmc^1(\mgg)$. \end{pro}

The conditions above are structural properties of a Lie algebra admitting ad-invariant metrics, but they are independent of the metric. 

\begin{rem} Given an arbitrary Lie algebra, the ideal  $\Theta_{\mv_1,\ldots,\mv_s}$ varies depending on the decomposition of $\mgg$ taken. 
One can find examples of Lie algebras $\mgg$ for which $\Theta=0$ for any decomposition, but not admitting ad-invariant metrics.  
\end{rem}

\begin{rem} \label{rem.degenerate} When a bilinear form $\bil$ on $\mgg$ is symmetric and invariant (possibly degenerate) some of the reasonings above are still valid. Specifically, in this context, one has that for any $\mv\subset \mgg$, $\rmc_j(\mv)\subset \rmc^{j}(\mv)^\bot$ for all $j\geq 0$ and Eq. \eqref{eq.comm} holds as well as inclusions in \eqref{eq.cap} and \eqref{eq.theta}.
\end{rem}

A Lie algebra $\mgg$ is a $k$-step nilpotent Lie algebra if $\rmc^{k}(\mgg)=0$ and $\rmc^{k-1}(\mgg)\neq 0$. A 2-step nilpotent Lie algebra $\mgg$ is nonsingular if $\ad_X:\mgg\lra \mz$ is onto for all $X\in \mgg$; if this is not the case, we say that $\mgg$ is singular.

\begin{cor} \label{padinv} 
Let $\mgg$ be a 2-step nilpotent Lie algebra satisfying one of the following conditions:
\begin{enumerate}
\item $\mgg$ is nonsingular,
\item \label{padinv.2} $\mgg=\mv_1\oplus\mv_2\oplus\rmc^1(\mgg)$ where $[\mv_i,\mv_i]=0$, $i=1,2$ and $[\mv_1,\mv_2]= \rmc^1(\mn)$.
\end{enumerate}
Then $\mgg$ does not admit ad-invariant metrics.
\end{cor}

\begin{proof} 1. Let $\mgg$ be a 2-step nilpotent Lie algebra, then $0\neq\rmc^1(\mgg)\subset \mz$. If $\mgg$ is also nonsingular then $\mz=\ad_X(\mgg)=[X,\mgg]\subset [\mz(X),\mgg]$, for all $X\in\mgg$. Therefore $\mz\subset\bigcap_{X\in\mgg}[\mz(X),\mgg]$ and  \eqref{padinv.cap} does not hold. 

2. The hypothesis imply $\mv_i\subset \mz(\mv_i)$ since $[\mv_i,\mv_i]=0$ for $i=1,2$. Thus 
$$0\neq\rmc^1(\mn)=[\mv_1,\mv_2]\subset [\mz(\mv_1),\mv_2]=[\mz(\mv_1),\mv_2\oplus\mv_1]=[\mz(\mv_1),\mgg];$$ analogously  $\rmc^1(\mn)\subset [\mz(\mv_2),\mgg]$. Thus $\Theta_{\mv_1,\mv_2}\neq 0$ and now \eqref{padinv.theta} is not satisfied.
\end{proof}

\begin{rem} Heisenberg-Reiter Lie algebras \cite{T} are 2-step nilpotent Lie algebras $\mgg$ having a decomposition $\mgg=\mv_1\oplus \mv_2\oplus \mz$ where the Lie bracket verifies $[\mv_i,\mv_i]=0$, $i=1,2$. Corollary \ref{padinv} assures that Heisenberg-Reiter Lie algebras verifying $\mz=\rmc^1(\mn)$ cannot be endowed with an ad-invariant metrics.
\end{rem}

\section{Nilpotent Lie algebras associated with graphs}\label{sec.graphs}

Let $G=(V,E)$ be a finite simple graph where $V$ is the set of vertices and $E$ is the nonempty set of edges, that is, $E$ is a collection of unordered pairs of distinct vertices. An edge given by the unordered pair $v,w\in V$ is denoted $e=v w$ and we say that $e$ is incident on $v$ and $w$. Given $v\in V$ the set of neighbors of $v$ is $N(v)=\{w\in V:\,vw\in E\}$ and the degree of $v$ is $d(v)=|N(v)|$. A subset $W\subseteq V$ is a vertex-covering of $G$ if every edge of $G$ is incident on at least one vertex in $W$. Set $V_0=\{v\in V:\,d(v)=0\}$ and $V_1=\{v\in V:\,d(v)\geq 1\}$, so $V$ is the disjoint union of $V_0$ and $V_1$, with $V_1$ nonempty.
Let $\mathcal V_i$ be the real vector space having $V_i$ as a basis and define $\mathcal V=\mathcal V_0\oplus \mathcal V_1$. Let $\mathcal U$ be the subspace of $\Lambda^2\mathcal V_1$ spanned by $\{v\wedge w:\,v,w\in V, \,vw\in E\}$. 

The Lie algebra $\mn_G$ associated to $G$ is $\mn_G=\mathcal V\oplus \mathcal U$ where the Lie bracket is obtained by extending the following rules by skew-symmetry. If $v,w\in V$ then $[v,w]=v\wedge w$ if $vw\in E$ and $0$ otherwise, and $[v,u]=0$ for all $v\in V$ and $u\in\mathcal U$. The dimension of $\mn_G$ is $|V|+|E|$ and it is always 2-step nilpotent. The commutator is $\rmc^1(\mn_G)=\mathcal U$ and its center is $\mz=\mathcal V_0\oplus \mathcal U$. 

Let $\tilde G$ be the subgraph of $G$ with $\tilde G=(V_1,E)$. Clearly, $\mn_G\simeq \R^s\oplus\mn_{\tilde{G}}$ where $s=|V_0|$ (possibly $s=0$). Note that  $\rmc^1(\mn_{\tilde G})=\mathcal U$ is the center of $\mn_{\tilde G}$. Moreover, if $\tilde G_1, \ldots,\tilde G_r$ are the connected components of $\tilde G$ (possibly $r=1$) then $\mn_G\simeq \R^s\oplus \mn_{\tilde G_1}\oplus\cdots \oplus\mn_{\tilde G_r}$. For further notations, definitions and properties of graphs we refer to \cite{Bol}.

\begin{example}\label{exc3} The $3$-cycle graph is $C_3=(V,E)$ with $V=\{v_1,v_2,v_3\}$ and edges $E=\{v_1v_2,v_2v_3,v_3v_1\}$. The Lie algebra $\mn_{C_3}$ is $6$-dimensional and isomorphic to the free $2$-step nilpotent Lie algebra on $3$-generators $\mn_{3,2}$. Any direct sum $\R^s\oplus \mn_{3,2}\oplus\stackrel{r}{\cdots} \oplus\mn_{3,2}$ is the Lie algebra associated with the graph having $s$ isolated vertices and $r$ nontrivial connected components each of which is a $3$-cycle. 

The Lie algebra $\mn_{3,2}$ admits ad-invariant metrics \cite{dBO,FS} and also does any direct sum $\R^s\oplus \mn_{3,2}\oplus\cdots \oplus\mn_{3,2}$ by taking product metrics.
\end{example}

The next lemma describes subspaces of $\mn_G$ considered in the previous section. Given $v\in V$, we say that $v$ is in a $3$-cycle if there exists a subgraph $H=(W,F)$ of $G$ isomorphic to $C_3$ such that $v\in W$. If this is not the case, we say that $v$ is not in a $3$-cycle.

\begin{lm}
Let $\mn_G$ be the Lie algebra associated with $G=(V,E)$ and let $v\in V$. Then $\mz(v)=\mz\oplus \span_\R V\backslash N(v)$ and if $v$ is not in a 3-cycle then $V\bs N(v)$ is a vertex-covering of $G$ and $[\mz(v),\mn_G]=\rmc^1(\mn_G)$.
\end{lm}

\begin{proof}
From the Lie bracket defining $\mn_G$,  one has $\mz\oplus \span_\R V\backslash N(v)\subset\mz(v)$. The fact $$\mn_G=\mz\oplus \span_\R V\backslash N(v)\oplus \span_\R N(v)$$ together with $[w,v]\neq 0$ for all $w\in N(v)$ yield to $\mz(v)=\mz\oplus \span_\R V\backslash N(v)$. 

Let $v$ be a vertex in $V$ not in a 3-cycle. Suppose $d(v)=0$, then $N(v)=\emptyset$ and $V\bs N(v)=V$ is trivially a vertex-covering of $G$. Assume now that $d(v)>0$ and suppose there is an edge $e\in E$ not incident on any vertex in $V\bs N(v)$. Then $e=v_1v_2$ with $v_1,v_2\in N(v)$ and therefore the subgraph with vertices $v,v_1,v_2$ is a 3-cycle in $G$, contradicting the hypothesis.
Hence $V\bs N(v)$ is a vertex-covering of $G$. 

Clearly $[\mz(v),\mn_G]\subset \rmc^1(\mn_G)=\mathcal U$. Let $v_1\wedge w$ a basic element of $\mathcal U$, then $e=v_1w\in E$. Since $V\bs N(v)$ is a vertex-covering we may assume that $v_1\in V\bs N(v)$. Thus $v_1\wedge w\in[\span_\R V\bs N(v),\mn_G]$ and $\mathcal U\subset [\mz(v),\mn_G]$. Therefore $[\mz(v),\mn_G]=\rmc^1(\mn_G)$. 
\end{proof}

\begin{lm}
If $G=(V,E)$ is such that $\mn_G$ admits an ad-invariant metric, then \begin{equation}\label{eq.tree}|E|=|V_1|\end{equation} and every vertex in $V_1$ is in a 3-cycle. \end{lm}

\begin{proof} The 2-step nilpotent Lie algebra $\mn_G$ satisfies: $\dim \mz=|V_0|+|E|$, $\dim \rmc^1(\mn_G)=|E|$ and $\dim \mn_G=|V_0|+|V_1|+|E|$. If $\mn_G$ admits an ad-invariant metric, Eq. \eqref{eq.dim} implies 
 $$|V_0|+|E|+|E|= |V_0|+|V_1|+|E|,$$ so
 $G$ satisfies Eq. \eqref{eq.tree}.
 
 Let $v\in V_1$ such that $v$ is not in a $3$-cycle. Then $[\mz(v),\mn_G]=\rmc^1(\mn_G)\subset (\R v)^\bot$ because of \eqref{eq.comm} and thus $v\in \rmc^1(\mn_G)^\bot=\mz=\mathcal V_0\oplus \mathcal U$ which is a contradiction.
\end{proof}

\begin{teo} \label{teo.classgraph}
The Lie algebra $\mn_G$ associated with a graph $G$ admits an ad-invariant metric if and only if $\mn_G$ is isomorphic to a direct sum $\R^s\oplus\mn_{3,2}\oplus\cdots\oplus\mn_{3,2}$. Equivalently, $\mn_G$ admits an ad-invariant metric if and only if each connected component of $G$ is a 3-cycle or an isolated vertex. 
\end{teo}

\begin{proof}
Example \ref{exc3} shows that any direct sum of $\mn_{3,2}$ with itself or with abelian factors admits an ad-invariant metric. So we focus on proving the converse.

Let $G$ be a graph such that $\mn_G$ admits an ad-invariant metric. 
Let $H=(W,F)$ be a connected component of $\tilde G=(V_1,E)$; any vertex in $H$ is in a 3-cycle by the previous lemma. Since  $H$ is connected we know that $|F|\geq |W|-1$. If $|F|=|W|-1$ then $H$ is a tree but in this case the vertices in $W$ are not in a $3$-cycle. So we have $|F|\geq |W|$ which togheter with \eqref{eq.tree} (for $G$) imply $|F|=|W|$. Thus $H$ is obtained from a tree by adding a sole edge. Since any element in $W$ is in a cycle we get that $H$ itself is a 3-cycle. \end{proof}

\section{Nilradicals of parabolic subalgebras}\label{sec.nilrad}

\subsection{Structure of nilradicals of parabolic subalgebras}
Let $\mgg$ be a split real simple 
Lie algebra with triangular decomposition $\mgg=\mgg^{-}+\mh+\mgg^+$ associated to a positive root system $\Delta^+$ with simple roots $\Pi$ and let 
$\mb=\mh+\mgg^+$ be the corresponding Borel subalgebra of $\mgg$.
Given $\gamma\in\Delta$, $H_\gamma\in\mh$ is its dual element through the Killing form $\kappa$ of $\mgg$, and if $\delta\in\Delta$ denote $(\gamma,\delta)=\kappa(H_\gamma,H_\delta)$.
As usual, if $\gamma\in\Delta$, then $X_\gamma$ denotes an arbitrary root vector in the root space $\mgg_\gamma$ and if $\alpha\in\Pi$,  $\coord_{\alpha}(\gamma)$ denotes the $\alpha$-coordinate of $\gamma$ 
when it is expressed as a linear combination of simple roots. Let $\gamma_{\max}$ denote the unique maximal root of $\Delta^+$. 

The set of parabolic Lie subalgebras of $\mgg$ containing $\mb$ is parametrized by 
the subsets of the set of simple roots $\Pi$ as follows.
Given a subset $\Pi_0\subset\Pi$, the corresponding parabolic subalgebra of $\mgg$
is $\mpp\simeq \mgg_1\ltimes \mn$ where
$\mgg_1=\mgg_1^-\oplus \mh\oplus \mgg_1^+$ and  
\begin{eqnarray}
\Delta_1^+ &=& \{\gamma\in\Delta^+: \coord_{\alpha}(\gamma)=0\text{ for all }\alpha\in\Pi_0\},\nonumber\\
\Delta_{\mn}^+ &=& \{\gamma\in\Delta^+: \coord_{\alpha}(\gamma)\neq0\text{ for some }\alpha\in\Pi_0
\},\nonumber\\
\mgg_1^+&=& \mbox{ subspace of $\mgg^+$ spanned by the vectors }X_\gamma \mbox{  with } \gamma\in \Delta_1^+,\nonumber\\
\mgg_1^-&=& \mbox{ subspace of $\mgg^-$ spanned by the vectors }X_{-\gamma} \mbox{  with } \gamma\in \Delta_1^+,\nonumber\\
\mn&=& \mbox{ subspace of $\mgg^+$ spanned by the vectors }X_\gamma \mbox{  with } \gamma\in \Delta_{\mn}^+.\nonumber
\end{eqnarray}
 
On the one hand, $\mgg_1$ is reductive and it can be viewed as
a Lie subalgebra of $\text{Der}(\mn)$ via $\ad_{\mgg}$. 

On the other hand, $\mn$ is nilpotent and its lower central series (which coincides, after transposing the 
indexes, with the upper central series) can be described as follows \cite[Theorem 2.12]{Ko1}. Given $\gamma\in\Delta$, let 
\[
 o(\gamma)=\sum_{\alpha\in\Pi_0} \coord_{\alpha}(\gamma)
\]
(in particular $\gamma\in\Delta_{\mn}^+$ if and only if $ o(\gamma)>0$) and, for $i\in\Z$, let 
\[
 \mg_{(i)}=\bigoplus_{\begin{smallmatrix}
               \gamma\in\Delta \\
               o(\gamma)=i
              \end{smallmatrix}} \mg_{\gamma}.
\]
If  
$\mn=\rmc^0(\mn)\supset\rmc^1(\mn)\supset\dots\supset\rmc^{k-1}(\mn)\supset\rmc^{k}(\mn)=0$ is
the lower central series of $\mn$, then
\begin{equation}\label{eq.lcs}
 \rmc^j(\mn)=\bigoplus_{i=j+1}^{k}\mg_{(i)},\;\; k=o(\gamma_{\text{max}})
 \text{ and $\rmc^{k-1}(\mn)=\mg_{(k)}$ is the center of $\mn$}.
\end{equation}
\begin{rem}\label{rem.abelian} It follows from this description of the lower central series that 
the nilradical $\mn$ is abelian if and only if $\Pi_0=\{\alpha\}$ and $\coord_{\alpha}(\gamma_{\max})=1$.
\end{rem}

For each $i\geq 2$ one has \cite{Ko1}
\begin{equation}
[\mgg_{(1)},\mgg_{(i-1)}]=\mgg_{(i)}.
\label{kostant}
\end{equation} Let $\gamma\in\triangle_\mn^+$ be such that $o(\gamma)\geq 2$ and let $\alpha\in\Pi_0$. Then Eq. \eqref{kostant} implies that at least one of the following conditions hold:
\begin{itemize}
\item[(i)] $\gamma=\delta+\beta$ where $\delta,\beta\in\Delta_\mn^+$ and 
$\coord_\alpha(\delta)=0$;
\item[(ii)] $\gamma=\tilde\gamma+\beta$ where $\tilde\gamma,\beta\in\Delta_\mn^+$ and 
$o(\beta)=1$ and $\coord_\alpha(\beta)=1$, $\coord_\alpha(\tilde\gamma)>0$.
\end{itemize}
Using these decompositions inductively we reach the following general results on the reduced root system associated to a parabolic subalgebra. 

\begin{pro}
Assume $|\Pi_0|\geq 2$ and let $\gamma\in\Delta_\mn^+$ be such that $X_\gamma\in\rmc^{k-1}(\mn)$. Then for any $\alpha\in\Pi_0$ there exists some $t\geq 1$ and roots 
$\delta,\beta_1,\ldots,\beta_t\in\Delta_\mn^+$ such that
\begin{equation}
\label{gammadecomp}
\gamma=\delta+\beta_t+\beta_{t-1}+\ldots+\beta_1
\end{equation}
and also 
$\coord_\alpha(\delta)=0$, $\coord_\alpha(\beta_j)\neq 0$ for all $j=1,\ldots,t$
 and any partial sum $\delta+\beta_t+\beta_{t-1}+\ldots+\beta_{t-s}$, $s=0,\ldots,t-1$, is a root in $\Delta_\mn^+$.
\label{progammadecomp}
\end{pro}

\begin{rem} \label{rem.gammadecomp} When $\gamma$ is as in \eqref{gammadecomp} one has $\sum_{j=1}^t\coord_\alpha(\beta_j)=\coord_\alpha(\gamma)$ and thus $t\leq \coord_\alpha (\gamma)$.
\end{rem}

\begin{proof} 
Let $\gamma\in\triangle_\mn^+$ such that $X_\gamma\in\rmc^{k-1}(\mn)$, then $\coord_{\tilde\alpha}(\gamma)=\coord_{\tilde\alpha}(\gamma_{\max})$ for all $\tilde\alpha\in\Pi_0$ and $o(\gamma)=k\geq 2$.

Let $\alpha\in\Pi_0$ and assume $\gamma$ satisfies (i) above, clearly \eqref{gammadecomp} holds for $t=1$. Assume this is not the case and $\gamma=\gamma_1+\beta_1$ with $o(\beta_1)=1$, $\coord_\alpha(\beta_1)=1$ and $\coord_\alpha(\gamma_1)>0$. Then $o(\gamma_1)\geq 2$ since $|\Pi_0|\geq 2$ so we can decompose $\gamma_1$ similarly.

If $\gamma_1$ verifies (i) then $\gamma_1=\delta+\beta_2$ being $\coord_\alpha(\delta)=0$. Hence $\gamma=\delta+\beta_2+\beta_1$ and \eqref{gammadecomp} holds with $t=2$. Otherwise $\gamma_1=\gamma_2+\beta_2$ with $o(\beta_2)=1$, $\coord_\alpha(\beta_2)=1$ and $\coord_{\alpha}(\gamma_2)>0$; thus $\gamma=\gamma_2+\beta_2+\beta_1$  and again $o(\gamma_2)\geq 2$, so we decompose $\gamma_2$ and so on.

Repeating this procedure one gets that $\gamma_j$ satisfies (i) for some $j\geq 1$. In fact, denote $d=\coord_\alpha(\gamma)$ and assume that $\gamma$ is decomposed using (ii) after $d-1$ steps, obtaining
$\gamma=\gamma_{d-1}+\beta_{d-1}+\beta_{d-2}+\ldots+\beta_1$ with
$\coord_\alpha(\beta_j)=1$ and $o(\beta_j)=1$ for all $j$, and $\coord_\alpha(\gamma_{d-1})>0$. As before,   $|\Pi_0|\geq 2$ implies $o(\gamma_{d-1})\geq 2$. Suppose $\gamma_{d-1}$ satisfies (ii), that is, $\gamma_{d-1}=\gamma_{d}+\beta_d$ where $\beta_d, \gamma_d\in\Delta_\mn^+$ with $o(\beta_d)=1$, $\coord_\alpha(\beta_d)=1$ and $\coord_\alpha(\gamma_{d})>0$. But then $\gamma=\gamma_{d}+\beta_{d}+\beta_{d-1}+\ldots+\beta_1$ with $\sum_{j=1}^d\coord_\alpha(\beta_j)=d=\coord_\alpha(\gamma)$, contradicting $\coord_\alpha(\gamma_{d})>0$. Therefore $\gamma_{d-1}$ verifies (i) above and the proposition holds.
\end{proof}

\begin{lm}\label{lm.subsroot}
Let $\gamma\in\Delta^+$ and $\alpha\in\Pi$ such that $\gamma\neq \alpha$.
\begin{enumerate}
\item If $\gamma+\alpha\notin \Delta^+$, then 
$\gamma-\alpha \in \Delta^+$ if and only if $(\gamma,\alpha)>0$.
\item If $\gamma-\alpha\notin \Delta^+$, then
$\gamma+\alpha\in \Delta^+$ if and only if $(\gamma,\alpha)<0$.
\end{enumerate}
\end{lm}

\begin{proof}
The $\alpha$-string through $\gamma$ is $\gamma+n\alpha$ where $p\leq n\leq q$ \cite[Proposition 8.4]{Hu} and \begin{equation}
p+q=-2\frac{(\gamma,\alpha)}{(\alpha,\alpha)}.\label{eq.string}
\end{equation} Assume $\gamma+\alpha\notin\Delta^+$, then $q=0$ and from the equation above, $(\gamma,\alpha)>0$ if and only if $p\leq -1$, or equivalently, $\gamma-\alpha\in\Delta^+$.
A similar argument proves the second assertion.

\end{proof}

This lemma and the description of the lower central series of $\mn$ in Eq. \eqref{eq.lcs} give a strategy to compute a basis (and thus the dimension) of both $\rmc^{k-1}(\mn)$ and $\rmc^1(\mn)$, as we describe in the following example.

We denote $\lela\lela\gamma,\alpha\rira\rira=2\frac{(\gamma,\alpha)}{(\alpha,\alpha)}$ which are the integers in the Cartan matrix when $\gamma$ and $\alpha$ simple roots. Notice that $(\gamma,\alpha)$ and $\lela\lela\gamma,\alpha\rira\rira$ either have the same sign or are simultaneously zero, since $(\alpha,\alpha)>0$. 

\begin{example} \label{exampleE} Consider the simple Lie algebra $E_6$ where $\Pi=\{\gamma_1,\ldots,\gamma_6\}$ ordered as in \cite[Theorem 11.4]{Hu}. Then $\gamma_{\max}=\gamma_1+2\gamma_2+2\gamma_3+3\gamma_4+2\gamma_5+\gamma_6$. Assume $\Pi_0=\{\alpha\}$; in particular $\mn$ is abelian if $\alpha=\gamma_1$ or $\alpha=\gamma_6$ (see Remark \ref{rem.abelian}). Below we assume $\mn$ is nonabelian.

Clearly $\gamma_{\max}+\gamma_i$ is never a root for $i=1,\ldots,6$. Moreover $\lela\lela\gamma_{\max},\gamma_i\rira\rira=1$ if $i=2$ and zero otherwise. Thus $\rmc^{k-1}(\mn)=\R X_{\gamma_{\max}}$ if $\alpha=\gamma_2$. 

If $\alpha\neq \gamma_2$ then $X_{\gamma_{\max}}$ and $X_{\gamma_{\max}-\gamma_2}$ are elements in $\rmc^{k-1}(\mn)$. Notice that $\gamma_{\max}-2\gamma_2$ is not a root since in \eqref{eq.string} we have $q=0$ and $p=-1$. Also $\lela\lela\gamma_{\max}-\gamma_2,\gamma_i\rira\rira=1$ if $i=4$ and $\leq 0$ otherwise. Hence $\rmc^{k-1}(\mn)$ is either $\R X_{\gamma_{\max}}\oplus\R X_{\gamma_{\max}-\gamma_2}$ if $\alpha=\gamma_4$ \linebreak or contains the subspace $\R X_{\gamma_{\max}}\oplus \R X_{\gamma_{\max}-\gamma_2}\oplus \R X_{\gamma_{\max}-(\gamma_2+\gamma_4)}$ if $\alpha\neq\gamma_4$. Following this reasoning one obtains
\begin{equation}\dim \rmc^{k-1}(\mn)=\left\{\begin{array}{l}
1 \quad \mbox{ if }\Pi_0=\{\gamma_2\},\\
2 \quad \mbox{ if }\Pi_0=\{\gamma_4\},\\
5 \quad \mbox{ if }\Pi_0=\{\gamma_3\},\\
4 \quad \mbox{ if }\Pi_0=\{\gamma_5\}.
\end{array}\right.\label{dimE6}\end{equation}

In order to give a basis of $\mn/\rmc^{1}(\mn)$ we use 2. in Lemma \ref{lm.subsroot}. Assume $\alpha=\gamma_2$, then by Eq. \eqref{eq.lcs} $X_\gamma$ defines a nonzero element in $\mn/\rmc^{1}(\mn)$ if $\coord_{\gamma_2}(\gamma)=1$. Also, for $\gamma_i\neq \gamma_2$,  $\lela\lela\gamma_2,\gamma_i\rira\rira>0$ only when $i=4$ hence $\gamma_2,\gamma_2+\gamma_4$ induce linearly independent elements in $\mn/ \rmc^{1}(\mn)$. Thus $\dim \mn/\rmc^{1}(\mn)\geq 2>\dim \rmc^{k-1}(\mn)$ in the case $\alpha=\gamma_2$.

Similar arguments give $\dim \mn/\rmc^{1}(\mn)>\dim \rmc^{k-1}(\mn)$ for any other case in \eqref{dimE6}. 
\end{example}

\subsection{Nilradicals admitting ad-invariant metrics}
We continue with the notation in the previous subsection. We identify the nilradicals $\mn$ admitting ad-invariant metrics in terms of the simple Lie algebra $\mgg$ and the set $\Pi_0$ associated to $\mn$.

\begin{pro}\label{pro.2roots}
Let $\mgg$ be a split real form of a complex simple Lie algebra and let $\mn$ be the nilradical of a parabolic subalgebra of $\mgg$ associated to a set $\Pi_0$ of simple roots. If $|\Pi_0|\geq 2$ then $\mn$ does not admit ad-invariant metrics.
\end{pro}

\begin{proof} The Lie algebra $\mn$ is $k$-step nilpotent with $k\geq 2$ since $|\Pi_0|\geq 2$, $\mn$ decomposes as $\mn=\mgg_{(1)}\oplus \rmc^1(\mn)$ (see \eqref{eq.lcs}) and $\rmc^{k-1}(\mn)=\mz$. If $\mn$ admits an ad-invariant metric $\bil$, then $\rmc^{k-1}(\mn)=\rmc^1(\mn)^\bot$ by Proposition \ref{padinv.dim}. Below we show that  $\rmc^{k-1}(\mn)$ is also orthogonal to $\mgg_{(1)}$ and therefore $\rmc^{k-1}(\mn)\subset \rad=0$ which is a contradiction.

Let $X_\gamma\in\rmc^{k-1}(\mn)$ and $X_\mu\in\mgg_{(1)}$. There is some  $\alpha\in\Pi_0$ such that $\coord_\alpha(\mu)=1$. By Propostion \ref{progammadecomp} there exist some $t\geq 1$ and roots $\delta,\beta_j\in\Delta^+_\mn$ verifying $\coord_\alpha(\delta)=0$ and $\coord_\alpha(\beta_j)\neq 0$ for $j=1,\ldots,t$ such that
$$ \gamma=\delta+\beta_t+\ldots+\beta_1$$  and  $\delta+\beta_t+\ldots+\beta_{t-s}\in\Delta_\mn^+$ for all $s=0,\ldots,t-1$. In particular 
$$X_\gamma=[[[X_\delta,X_{\beta_t}],X_{\beta_{t-1}}],\cdots,X_{\beta_1}] \;\;\mbox{ (up to non-zero multiple)}.$$
The ad-invariance of the metric implies
\begin{eqnarray}
\lela X_\gamma,X_\mu\rira&=&\lela [[[X_\delta,X_{\beta_t}],X_{\beta_{t-1}}],\cdots,X_{\beta_1}],X_\mu\rira\nonumber\\
&=&\lela [[[X_\delta,X_{\beta_t}],X_{\beta_{t-1}}],\cdots,X_{\beta_2}],[X_{\beta_1},X_\mu]\rira\nonumber\\
&=&\lela X_\delta,[X_{\beta_t},\cdots,[X_{\beta_2},[X_{\beta_1},X_{\mu}]]]\rira. \label{equival}\end{eqnarray}

Notice that $\beta_t+\beta_{t-1}+\ldots+\beta_1+\mu$ is not a root in $\Delta_\mn^+$ since, when written as linear combination of simple roots in $\Pi$, its coefficient in $\alpha$ is greater than $\coord_\alpha(\gamma_{\max})$. Therefore  $[X_{\beta_t},\cdots,[X_{\beta_2},[X_{\beta_1},X_{\mu}]]]=0$ and from \eqref{equival}, $\lela X_\gamma,X_\mu\rira=0$. 
\end{proof}

The proof above motivated the result in Proposition \ref{pro.cap}. In fact any $\mn$ associated to $\Pi_0=\{\alpha_1,\ldots,\alpha_s\}$ admits the decomposition $\mn=\mv_1\oplus\cdots\oplus\mv_s\oplus\rmc^1(\mn)$ where $\mv_i$ is the span of root vectors $X_\beta$ with $\coord_{\alpha_i}(\beta)=1$ and $\coord_{\alpha_j}(\beta)=0$ for $j\neq i$. Fix $i=1,\ldots,s$ and lets assume that every $X_\gamma\in\rmc^{k-1}(\mn)$ satisfies \eqref{gammadecomp} with $t=1$. Then $X_\gamma=[X_\delta,X_\beta]$ with $X_\beta\in\mz(\mv_i)$ since $\coord_{\alpha_i}(\beta)=\coord_{\alpha_i}(\gamma_{\max})$ by Remark \ref{rem.gammadecomp}. In particular this implies $\rmc^{k-1}(\mn)\subset [\mz(\mv_i),\mn]$.
The proof above is a generalization of this fact for $t\geq 2$.

\begin{teo}\label{teo.classif}
Let $\mn$ be a nilradical of a parabolic subalgebra associated to a split real form of a simple Lie algebra $\mgg$. Then $\mn$ admits ad-invariant metric if and only if $\mn$ is abelian or it is isomorphic to either $\mn_{3,2}$ or $\mn_{2,3}$.
\end{teo}

\begin{proof} 
For a simple Lie algebra $\mgg$ we consider the numbering of simple roots $\Pi=\{\gamma_1,\gamma_2,\ldots,\gamma_n\}$ as in \cite[Theorem 11.4]{Hu}.

Let $\mn$ be the nilradical associated to $\Pi_0\subset \Pi$ and assume it admits an ad-invariant metric. Proposition \ref{pro.2roots} imply $\Pi_0=\{\alpha\}$; the relevant cases to consider are those for which $\coord_\alpha(\gamma_{\max})\geq 2$ since $\mn$ is abelian otherwise. We give the proof type by type:

$\bullet$ Type A. Here $\coord_\alpha(\gamma_{\max})=1$ for all $\alpha\in \Pi$.

$\bullet$ Type B. Notice that $\coord_{\gamma_1}(\gamma_{\max})=1$. Also if $\alpha=\gamma_n$ then $\mn$ is the free nilpotent Lie algebra $\mn_{,2}$ which admits ad-invariant metric if and only if $n=3$ \cite{dBO}. 

Let $\alpha=\gamma_l$ for some $2\leq l\leq n-1$, then $\mn$ is 2-step nilpotent and $\mn=\mv_1\oplus\mv_2\oplus\mv_3\oplus\rmc^1(\mn)$ where
$$\begin{array}{rclrcl}
\mv_1&=&\span\{X_{\epsilon_i-\epsilon_j}:\;1\leq i\leq l<j\},&
\mv_3&=&\span\{X_{\epsilon_i}:\;1\leq i\leq l\},\\
\mv_2&=&\span\{X_{\epsilon_i+\epsilon_j}:\;1\leq i\leq l<j\},&
\rmc^1(\mn)&=&\span\{X_{\epsilon_i+\epsilon_j}:\;1\leq i<j\leq l\}.
\end{array}$$
Moreover $\rmc^1(\mn)=\mz=[\mv_1,\mv_2]=[\mv_3,\mv_3]$ and $[\mv_i,\mv_3]=0$ for $i=1,2$. This implies $\mz\subset [\mz(\mv_i),\mn]$ for $i=1,2,3$. Thus $\Theta_{\mv_1,\mv_2,\mv_3}\neq 0$ for this decomposition and \eqref{padinv.theta} does not hold and $\mn$ cannot be endowed with an ad-invariant metric.

$\bullet$ Type C. Here $\coord_{\gamma_n}(\gamma_{\max})=1$, so consider $\alpha=\gamma_l$ for some $l=1,\ldots,n-1$. In this case, $\mn=\mv_1\oplus\mv_2\oplus\rmc^1(\mn)$ where
$$\begin{array}{rclrcl}
\mv_1&=&\span\{X_{\epsilon_i-\epsilon_j}:\;1\leq i\leq l<j\},&
\rmc^1(\mn)&=&\span\{X_{\epsilon_i+\epsilon_j}:\;1\leq i<j\leq l\}.\\
\mv_2&=&\span\{X_{\epsilon_i+\epsilon_j}:\;1\leq i\leq l<j\},&&&
\end{array}$$
Also, $\mz=\rmc^{1}(\mn)=[\mv_1,\mv_2]$ and $0=[\mv_i,\mv_i]$, for $i=1,2$. Thus Corollary \ref{padinv} implies that $\mn$ does not admit ad-invariant metric. 

$\bullet$ Type D.  Notice that $\coord_{\gamma_i}(\gamma_{\max})=1$ if $i=1,n-1,n$, so assume $\alpha=\gamma_l$ for some $l=2,\ldots,n-2$. In this case,  $\mn=\mv_1\oplus\mv_2\oplus\rmc^1(\mn)$ where
$$\begin{array}{rclrcl}
\mv_1&=&\span\{X_{\epsilon_i-\epsilon_j}:\;1\leq i\leq l<j\},&
\rmc^1(\mn)&=&\span\{X_{\epsilon_i+\epsilon_j}:\;1\leq i<j\leq l\}.\\
\mv_2&=&\span\{X_{\epsilon_i+\epsilon_j}:\;1\leq i\leq l<j\},&&&
\end{array}$$
Again $\mz=\rmc^{1}(\mn)=[\mv_1,\mv_2]$ and $0=[\mv_i,\mv_i]$, for $i=1,2$ so $\mn$ does not admit ad-invariant metrics.

$\bullet$ Type G. When $\mgg=\mgg_2$ straightforward computations show that $\mn=\mn_{2,3}$ if $\alpha=\gamma_1$ so it admits ad-invariant metric \cite{dBO}. To the contrary, when $\alpha=\gamma_2$ the nilradical $\mn$ verifies $\dim \mz=1$ and $\dim \mn-\dim\rmc^1(\mn)\geq 2$ so Eq. \eqref{eq.dim} does not hold and $\mn$ cannot be endowed with an ad-invariant metric.

$\bullet$ Types E and F. Every nilradical is nonabelian. The argument given in Example \ref{exampleE} applies in these cases and one gets $\dim\rmc^{k-1}(\mn)+\dim\rmc^{1}(\mn)>\dim \mn$, independently of $\alpha\in\Pi$. Thus $\mn$ does not admit ad-invariant metrics by Proposition \ref{pro.adinv}
\end{proof} \bigskip

\noindent {\bf Acknowledgment.} Special thanks to Leandro Cagliero for the formal statement and guidelines for the proof of Proposition \ref{progammadecomp}.

\end{document}